\documentclass[twoside,reqno]{amsart}

\usepackage{amsmath,amssymb,amsbsy,amstext,amsxtra,latexsym}
\usepackage{graphicx}
\usepackage{subfig}
\usepackage{enumerate}
\usepackage[all]{xy}
\usepackage{cancel}
\usepackage{chngpage}
\usepackage{mathtools}

\usepackage[colorlinks=true]{hyperref}

\hypersetup{%
pdftitle={Infinitesimal deformations of complements of plumbings of rational curves},%
pdfauthor={Dongsoo Shin, Department of Mathematics, Chungnam National University, Korea},%
pdfkeywords={Infinitesimal deformation, plumbing},%
citecolor=black,%
filecolor=black,%
linkcolor=black,%
urlcolor=black,%
}

\newtheorem{theorem}{Theorem}[section]

\newtheorem{lemma}[theorem]{Lemma}
\newtheorem{proposition}[theorem]{Proposition}

\theoremstyle{definition}

\theoremstyle{remark}

\newtheorem{remark}[theorem]{Remark}

\numberwithin{equation}{section}

\newcommand{\udc}[2]{\ensuremath \underset{#2}{\overset{#1}{\circ}}}

\newcommand{\linsys}[1]{\ensuremath \left| #1 \right|}

\DeclareMathOperator{\image}{im}

\DeclareMathOperator{\ext}{Ext}
\DeclareMathOperator{\divisorgroup}{Div}
\DeclareMathOperator{\jacobian}{Jac}
\DeclareMathOperator{\divisor}{div}
\DeclareMathOperator{\symmetric}{Sym}
\DeclareMathOperator{\supp}{Supp}
\DeclareMathOperator{\picard}{Pic}

\def\sheaf#1{\ensuremath \mathcal#1}

\begin{document}

\title[Infinitesimal deformations of complements of plumbings]{Infinitesimal deformations of complements of plumbings of rational curves}

\author{Dongsoo Shin}

\address{Department of Mathematics, Chungnam National University, 79 Daehak-ro, Yuseong-gu, Daejeon 305-764, Korea}

\email{dsshin@cnu.ac.kr}

\date{\today}

\subjclass[2000]{14B12}

\keywords{infinitesimal deformation, plumbing}

\thanks{This research was supported by Basic Science Research Program through the National Research Foundation of Korea(NRF) funded by the Ministry of Education, Science and Technology(2010-0002678).}

\begin{abstract}
We construct infinitesimal deformations on an open domain of a smooth projective surface given by a complement of plumbings of disjoint linear chains of smooth rational curves. We show that the infinitesimal deformations are not small deformations, that is, they change the complex structure away from the boundary of the domain.
\end{abstract}

\maketitle

Let $Z$ be a smooth projective surface and $H$ an ample divisor on $Z$. Let $D_i = \sum_{j=1}^{m_i} C_{ij}$ ($i=1,\dotsc,k$) be a linear chain of rational curves $C_{ij}$ with $C_{ij}^2 = -b_{ij}$ ($b_{ij} \ge 2$) on $Z$; that is, the dual graph of $D_i$ is given as follows:
    \begin{equation*}
    D_i = \udc{C_{i1}}{-b_{i1}}-\udc{C_{i2}}{-b_{i2}}- \cdots -\udc{C_{im_i}}{-b_{im_i}}.
    \end{equation*}
A \emph{plumbing} of a linear chain $D_i$ is a union of small tubular neighborhoods of each rational curves $C_{ij}$ ($j=1,\dotsc,m_i$). Assume that the supports of $D_i$'s are disjoint and each divisor $D_i$ can be contracted to an isolated rational singular point. Let $f : Z \to X$ be the map contracting $D_i$'s.

In this paper we construct infinitesimal deformations on the complement of the union of certain plumbings of $D_i$'s. In Section~\ref{section:tubular-neighborhood} we prove that, for each $i$, there is a plumbing, say $U_i$, of $D_i$ corresponding to the ample divisor $H$ on $Z$ such that $H^2(Z \setminus U_i, \sheaf{T_{Z \setminus U_i}})=0$. Furthermore we show that $H^2(Z \setminus U, \sheaf{T_{Z \setminus U}})=0$ where $U = \cup_{i=1}^{k} U_i$; Theorem~\ref{theorem:tubular-neighborhood}. Since $Z \setminus U$ is not compact, the vanishing does not imply that an infinitesimal deformation on $Z \setminus U$ is integrable. We leave the integrability question for future research. In Section~\ref{section:infinitesimal-deformations} we construct infinitesimal deformations on $Z \setminus U$ by restricting to $Z \setminus U$ meromorphic $1$-forms on $Z$ with poles along $C_{ij}$ of certain orders. The orders of poles are determined by $H$ uniquely up to constant multiple; Theorem~\ref{theorem:infinitesimal_deformation}.  We finally show that the infinitesimal deformations are not small deformations, that is, the infinitesimal deformations change the complex structure away from the boundary of $Z \setminus U$; Theorem~\ref{theorem:nontrivial}.

In order to construct such infinitesimal deformations, we apply a similar strategy in Takamura~\cite{Takamura-1996}, where he constructed infinitesimal deformations on a complement of a certain negative non-rational curve on a smooth surface of general type.

Throughout this paper we work over the field $\mathbb{C}$ of complex numbers.

\section{Holomorphic tubular neighborhood}
\label{section:tubular-neighborhood}

We first show that there is a special holomorphic tubular neighborhood for each divisor $D_i = \sum_{j=1}^{m_i} {C_{ij}}$ in the surface $Z$ with a vanishing cohomology condition; Proposition~\ref{theorem:tubular-neighborhood}.

\begin{lemma}\label{lemma:tubular-neighborhood}
Let $C$ be a smooth rational curve with $C^2 \le -1$ on $Z$. Then there is a holomorphic tubular neighborhood $V$ of $C$ in $Z$ such that $H^2(Z \setminus V, \sheaf{T_{Z \setminus V}})=0$.
\end{lemma}

\begin{proof}
We use a similar method in Takamura~\cite{Takamura-1996}. Let $A \in \linsys{lH}$ ($l \gg 0$) be a smooth irreducible hyperplane section of $Z$. Fix a base point $p \in C \cup A$ of the Abel-Jacobi map $\alpha: \divisorgroup(A) \to \jacobian(A)$. Let $V_p$ be a small open neighborhood of $p$ in $A$. Choose a nonzero section $s \in H^0(A, \sheaf{O_A}(A))$. Set $R = \divisor(s)$, which is an effective divisor on $A$ of degree $d=A^2$. Note that the fundamental domain of the complex torus $\jacobian(A)$ is bounded and $\alpha(V_p)$ is an open set (by shrinking $V_p$ if necessarily) which contains the origin in $\jacobian(A)$; hence we have $\alpha(\symmetric^{nd}{V_p}) = \jacobian(A)$ for sufficiently large $n$ because the Abel-Jacobi map $\alpha$ is a homomorphism of abelian groups. Therefore we have $\alpha(nR) \in \alpha(\symmetric^{nd}{V_p})$ for sufficiently large $n$. Since $\alpha^{-1}(\alpha(nR)) = \mathbb{P}H^0(A, \sheaf{O_A}(nA))$, we have
    \[\alpha^{-1}(\alpha(nR)) \cap \symmetric^{nd}{V_p} = \mathbb{P}H^0(A, \sheaf{O_A}(nA)) \cap \symmetric^{nd}{V_p} \neq \varnothing.\]

Choose an effective divisor $R' \in \mathbb{P}H^0(A, \sheaf{O_A}(nA)) \cap \symmetric^{nd}{V_p}$ and let $s' \in H^0(A, \sheaf{O_A}(nA))$ be a section such that $R'=\divisor(s')$. Note that $\supp(R') \subset V_p$. Since $A$ is a hyperplane section, we have $H^1(Z, \sheaf{O_Z}((n-1)A)=0$; hence the restriction map
    \begin{equation*}
    H^0(Z, \sheaf{O_Z}(nA)) \to H^0(A, \sheaf{O_H}(nA))
    \end{equation*}
is surjective. Therefore there is an smooth irreducible curve $A' \in \linsys{nA}$ of $Z$ whose restriction on $A$ is $R'$. Then $A \cap A' \subset V_p$. We may assume that $A'$ intersects $C$ transversely and that $p \not\in A' \cap C$ because $A'$ is also a hyperplane section.

Since $\ext^1(\sheaf{N_{C,Z}}, \sheaf{T_C}) = 0$, we have $\sheaf{T_Z}|_C = \sheaf{T_C} \oplus \sheaf{N_{C,Z}}$. Therefore there is a holomorphic tubular neighborhood $V$ of $C$ in $Z$. We may assume the followings: (i) $V \supset V_p$ by shrinking $V_p$ if necessarily, (ii) $V \supset A \cap A'$, and (iii) at least one component of $V \cap A$ (and also $V \cap A'$) is a fiber of $V$, where $V$ is considered to be a holomorphic disk bundle over $C$.

We remark the following observation: Let $N$ be a holomorphic line bundle over an open Riemann surface $S$ and let $N_0$ be obtained by deleting a tubular neighborhood of the zero section. Then $N_0$ is biholomorphic to $S \times B$ where $B$ is a complement of an open disk in $\mathbb{C}$ because any holomorphic line bundle over an open Riemann surface is trivial. Since both $S$ and $B$ are Stein, so is $N$. In particular, the boundary $\partial N_0$ is pseudoconvex.

In our case, the set $Z_0 = Z \setminus (A \cup V)$ (or $Z_0' = Z \setminus (A' \cup V$) is pseudoconvex in the affine variety $Z \setminus A$ (respectively, $Z \setminus A'$) by the above observation. Therefore $Z_0$ and $Z_0'$ are Stein. And $Z_0 \cap Z_0'$ is also stein. Since $A \cap A' \subset V_p$, we have $Z \setminus V = Z_0 \cup Z_0'$. By the Mayer-Vietoris sequence, we have the exact sequence
    \[H^1(Z_0 \cap Z_0', \sheaf{T_Z}) \to H^2(Z_0 \cup Z_0', \sheaf{T_Z}) \to H^2(Z_0, \sheaf{T_Z}) \oplus H^2(Z_0', \sheaf{T_Z}).\]
Therefore, according to the vanishing of higher cohomology of a Stein space, we have $H^2(Z \setminus V, \sheaf{T_{Z \setminus V}})=0$.
\end{proof}

\begin{theorem}\label{theorem:tubular-neighborhood}
There is a holomorphic tubular neighborhood $U_i$ for each $D_i=\sum_{j=1}^{m_i} C_{ij}$ in $Z$ such that $H^2(Y, \sheaf{T_{Y_i}})=0$ where $Y_i = Z \setminus U_i$. Furthermore, setting $Y=Z \setminus (U_1 \cup \dotsb \cup U_k)$, we have $H^2(Y, \sheaf{T_Y})=0$.
\end{theorem}

\begin{proof}
 By Lemma~\ref{lemma:tubular-neighborhood} there is a holomorphic neighborhood $W_{ij}$ of $C_{ij}$ such that $H^2(Y_{ij}, \sheaf{T_{Y_{ij}}})=0$ for each $j=1,\dotsc,m_i$ where $Y_{ij} = Z \setminus W_{ij}$. Set $U_i = \cup_{j=1}^{m_i} W_{ij}$ and $Y_i = Z-U_i = \cap_{j=1}^{m_i} Y_{ij}$. By the Mayer-Vietoris sequence and the induction on $m_i$ it is easy to show that $H^2(Y_i, \sheaf{T_{Y_i}})=0$. Since $Y=\cup_{i=1}^{k}{Y_i}$ it follows that $H^2(Y, \sheaf{T_Y})=0$.
\end{proof}

\begin{remark}
Since $Y = Z \setminus U$ is not compact, the vanishing does not imply that a infinitesimal deformation on $Z \setminus U$ is automatically integrable.
\end{remark}

\section{Infinitesimal deformations}
\label{section:infinitesimal-deformations}

Let $H$ be an effective ample divisor on $Z$. We construct infinitesimal deformations of the open surface $Y=Z \setminus U$ derived from $H$; Proposition~\ref{theorem:infinitesimal_deformation}.

\begin{lemma}\label{lemma:E-and-L}
There exists an effective divisor $L = x_0 H + E$ where
    \[E=\sum_{j=1}^{m_1}{x_{1j} C_{1j}} + \dotsb + \sum_{j=1}^{m_k}{x_{kj} C_{kj}}\]
satisfying the following properties:
    \begin{enumerate}[\rm (i)]
    \item $x_{ij} \ge 1$ for all $i$, $j$,

    \item $LC_{ij}=0$ for all $i$, $j$,

    \item $L$ descends to an ample divisor on $X$.
    \end{enumerate}
Furthermore a tuple of coefficients $(x_0, x_{11}, \dotsc, x_{km_k})$ is uniquely determined by $H$ up to a constant multiple.
\end{lemma}

\begin{proof}
For simplicity we first consider the case $k=1$. We denote $C_{1j}$ by $C_j$, $b_{1j}$ by $b_j$, and $m_1$ by $m$. Set $a_j = HC_j > 0$ for $j=1, \dotsc, m$. If $m=1$, set $L = b_1H + a_1C_1$. If $m=2$, set $L = (b_1b_2-1)H + (a_1b_2+a_2)C_1 + (a_1+a_2b_1)C_2$. For $m \ge 3$, the condition (i) would be interpreted as the system of linear equations
    \begin{equation}\label{equation:H}
    BX=0
    \end{equation}
where
    \begin{equation*}
    B=\begin{pmatrix}
    -b_1 & 1    &        &          &          & a_1     \\
    1    & -b_2 & 1      &          &          & a_2     \\
         &      & \ddots &          &          & \vdots  \\
         &      & 1      & -b_{m-1} & 1        & a_{m-1} \\
         &      &        & 1        & -b_m     & a_{m}
    \end{pmatrix}, \quad
    X=\begin{pmatrix}
    x_1 \\ x_2 \\ \vdots \\ x_m \\ x_0
    \end{pmatrix}.
    \end{equation*}
Since $B$ is of full rank, the solution of \eqref{equation:H} is unique up to constant multiple. Since $a_i, b_i > 0$, it is easy to show that there are positive integers $x_0, \dotsc, x_m$ satisfying \eqref{equation:H}. Therefore there is an effective divisor $L = x_0 H + x_1 C_1 + \dotsb + x_m C_m$ such that $LC_i=0$ for all $i=1,\dotsc,m$.

We now show that $L$ descends to an ample divisor on $X$. Let $E = x_1 C_1 + \dotsb + x_m C_m$. Consider the exact sequence
    \begin{equation*}
    H^0(Z, nx_0H) \xrightarrow{\phi} H^0(Z, nL) \xrightarrow{\psi} H^0(nE, nL|_{nE}) \to H^1(Z, nx_0H).
    \end{equation*}
Since $H$ is ample, $H^1(Z, nx_0H)=0$ for $n \gg 0$; hence $\psi$ is surjective for $n \gg 0$.

On the other hand, setting $V=f(U)$ and $p=f(D)$, $(U, D) \to (V, p)$ is the minimal resolution of the rational singularity $p$ of $X$. We may assume that $V \subset X$ is Stein and contractible by shrinking $U$ if necessary. Then $H^1(U, \sheaf{O_U})=H^2(U, \sheaf{O_U})=0$ and $D$ is a deformation retract of $U$, in particular the exponential sequence on $U$ gives an isomorphism $\picard(U) = H^2(U, \mathbb{Z})$. Therefore a line bundle $\sheaf{L}$ on $U$ is trivial if and only if $\sheaf{L} \cdot C_i=0$ for every irreducible component $C_i$ of $D$. Therefore the restriction of $L$ on $U$ is trivial and $nL$ is also trivial on $U$. In particular, the restriction map $H^0(nE, nL|_{nE}) \to H^0(D, nL|_D)$ is surjective.

Since $nL|_D$ is trivial, we can choose $s_0 \in H^0(Z, nL)$ such that $s_0|_D$ is a nonzero constant. On the other hand, since $H$ is ample, we may choose  $\{\widetilde{s}_1, \dotsc, \widetilde{s}_l\} \in H^0(Z, nx_0H)$ so that they gives an embedding $Z \hookrightarrow \mathbb{P}^{l-1}$ for $n \gg 0$. Set $s_i = \phi(s_i)$ ($i=1,\dotsc,l$). Since $s_i|_D=0$ for $i=1,\dotsc,l$, the map $\pi : Z \to \mathbb{P}^l$ defined by $\pi(x) = (s_0(x),s_1(x),\dotsc,s_l(x))$ contracts $D$ and gives an embedding of $Z \setminus D$; hence, the map $\pi$ is an embedding of $X$, which implies that $L$ descends to an ample divisor on $X$.

In case of $k \ge 2$, one may consider the following equation instead of \eqref{equation:H}: Setting $a_{ij} = HC_{ij} > 0$,
    \begin{equation*}
    \begin{pmatrix}
    B_1     &         &               &          & A_1\\
                & B_2 &               &          & A_2\\
                &         & \ddots &          & \vdots \\
                &         &                & B_k & A_k
    \end{pmatrix} X = 0
    \end{equation*}
where
    \begin{equation*}
    B_i=\begin{pmatrix}
    -b_{i1} & 1    &        &          &     \\
    1    & -b_{i2} & 1      &          &  \\
         &      & \ddots &          &       \\
         &      & 1      & -b_{i(m_i-1)} & 1 \\
         &      &        & 1        & -b_{m_i}
    \end{pmatrix},
    A_i=\begin{pmatrix}
    a_{i1} \\ a_{i2} \\ \vdots \\ a_{i(m_i-1)} \\ a_{im_i}
    \end{pmatrix},
    X=\begin{pmatrix}
    x_{11} \\ x_{12} \\ \vdots \\ x_{km_k} \\ x_0
    \end{pmatrix}.
    \end{equation*}
Then the proof is identical to the proof of the case $k=1$.
\end{proof}

Our infinitesimal deformations of $Y$ are obtained by restricting to $Y$ meromorphic $1$-forms on $Z$ with poles along $C_{ij}$. In order to construct such infinitesimal deformations we use a similar strategy in Takamura~\cite{Takamura-1996}.

\begin{theorem}\label{theorem:infinitesimal_deformation}
Let $E$ be the effective divisor in Lemma~\ref{lemma:E-and-L} corresponding to the ample divisor $H$. Then the restriction map $H^1(Z, \sheaf{T_Z}(nE)) \to H^1(Y, \sheaf{T_Y})$ is injective for any $n \gg 1$.
\end{theorem}

\begin{proof}
The divisor $L = x_0H + E$ descends to an ample divisor on $X$ by Lemma~\ref{lemma:E-and-L}. Let $\overline{L} = L|X$. Since $\overline{L}$ is ample, we may choose an irreducible smooth curve $C \in \linsys{n\overline{L}}$ for $n \gg 0$ such that $C$ does not pass through the singular point of $X$. We denote again by $C$ the inverse image of $C$ under the contraction map $f$. Then $C$ is linearly equivalent to $nL$ and $C \cap \supp(E) = \varnothing$.

Consider the exact sequence
    \begin{equation*}
    H^1(Z, \sheaf{T_Z}(nE-C)) \to H^1(Z, \sheaf{T_Z}(nE)) \xrightarrow{\alpha} H^1(C, \sheaf{T_Z}(nE) \otimes \sheaf{O_{C}}).
    \end{equation*}
Since $C-nE=nH$ is ample, we have
    \begin{equation*}
    \begin{split}
    H^1(Z, \sheaf{T_Z}(nE-C)) &= H^1(Z, \Omega_Z(K_Z + C - nE)) \\
    &=H^1(Z, \Omega_Z(K_Z + nH)) = 0.
    \end{split}
    \end{equation*}
Therefore the map $\alpha$ is injective. On the other hand, since $C \cap \supp(E) = \varnothing$, the map $\alpha$ factors through $H^1(Y, \sheaf{T_Y})$, i.e.,
    \begin{equation*}
    \alpha:  H^1(Z, \sheaf{T_Z}(nE)) \to H^1(Y, \sheaf{T_Y}) \to H^1(C, \sheaf{T_Z}(nE) \otimes \sheaf{O_{C}}).
    \end{equation*}
Therefore the restriction map $H^1(Z, \sheaf{T_Z}(nE)) \to H^1(Y, \sheaf{T_Y})$ is injective for any $n \gg 1$.
\end{proof}

\begin{remark}[Integrability]
Since $Y$ is not compact the vanishing of $H^2(Y, \sheaf{T_Y})$ does not necessarily imply the integrability of the infinitesimal deformations induced by $H^1(Z, \sheaf{T_Z}(nE))$. This question needs further investigation.
\end{remark}

The infinitesimal deformation $H^1(Z, \sheaf{T_Z}(nE))$ is nonempty; Proposition~\ref{proposition:grows_quadrically}. We use a similar method in Takamura~\cite{Takamura-2000}.

\begin{lemma}\label{lemma:h^1(T_Z(nE)|nE)}
The dimension $h^1(nE, \sheaf{T_Z}(nE) \otimes \sheaf{O_{nE}})$ grows at least quadratically in $n \gg 0$.
\end{lemma}

\begin{proof}
Since the supports of the divisors $D_i$ are disjoint we may assume that $k=1$. For simplicity we denote  $C_{1j}$ by $C_j$, $b_{1j}$ by $b_j$, $m_1$ by $m$, and $x_{1j}$ by $x_j$. We divide the proof into two cases.

Case 1: $m=1$. Let $C=C_1$, $b=b_1$, and $x=x_1$ for brevity. We first claim that $h^0(\sheaf{T_Z}(nE-lC)|_C)=0$ and $h^1(\sheaf{T_Z}(nE-lC)|_C)=2nx-2lb+b$ for $0 \le l \le nx$. Consider the exact sequence
    \[0 \to \sheaf{T_C}(nE-lC) \to \sheaf{T_Z}(nE-lC)|_C \to \sheaf{N_{C,Z}}(nE-lC) \to 0\]
induced from the tangent-normal bundle sequence. Since
    \begin{align*}
    \deg{\sheaf{T_C}(nE-lC)} &= -2-nxb+lb < 0,\\
    \deg{\sheaf{N_{C,Z}}(nE-lC)} &= -b-nxb+lb < 0
    \end{align*}
for $0 \le l \le nx$, it follows from the above exact sequence and the Riemann-Roch theorem that  $h^0(\sheaf{T_Z}(nE-lC)|_C)=0$ and  $h^1(\sheaf{T_Z}(nE-lC)|_C)=2nx-2lb+b$.

Consider the decomposition sequence
    \[0 \to \sheaf{T_Z}(nE-(l+1)C)|_{(nx-l-1)C} \to \sheaf{T_Z}(nE-lC)|_{(nx-l)C} \to \sheaf{T_Z}(nE-lC)|_C \to 0.\]
for $0 \le l \le nx_1-2$. Since $h^0(\sheaf{T_Z}(nE-lC)|_C)=0$, we have
    \[h^1(\sheaf{T_Z}(nE-lC)|_{(nx-l)C}) = h^1(\sheaf{T_Z}(nE-lC)|_C) + h^1(\sheaf{T_Z}(nE-(l+1)C)|_{(nx-l-1)C}).\]
Therefore
    \begin{equation*}
    \begin{split}
    h^1(\sheaf{T_Z}(nE)|_{nxC}) &= h^1(\sheaf{T_Z}(nE)|_C) + h^1(\sheaf{T_Z}(nE-C)|_{(nx-1)C}) \\
    &=\cdots\\
    &= \sum_{l=0}^{nx-1} h^1(\sheaf{T_Z}(nE-lC)|_C) \\
    &= \sum_{l=0}^{nx-1} (2nx-2lb+b) \\
    &\sim O(n^2).
    \end{split}
    \end{equation*}

Case 2: $m \ge 2$. The proof is similar to the case $m=1$. For the convenience of the reader, we briefly sketch the proof. We claim that $h^0(\sheaf{T_Z}(nE-lC)|_C)=0$ and $h^1(\sheaf{T_Z}(nE-lC)|_C)=2na_1x_0-2lb_1+b_1$ for $0 \le l \le \left[ \frac{na_1x_0+2}{b_1}\right] -1$. Set $\alpha = \left[ \frac{na_1x_0+2}{b_1}\right]$. Consider the exact sequence
    \[0 \to \sheaf{T_C}(nE-lC) \to \sheaf{T_Z}(nE-lC)|_C \to \sheaf{N_{C,Z}}(nE-lC) \to 0.\]
Since $-x_1b_1 + x_2 = -a_1x_0$ by \eqref{equation:H}, we have
    \begin{align*}
    \deg{\sheaf{T_C}(nE-lC)} &= -2 - nx_1b_1 + nx_2 + lb_1 = -2 - na_1x_0 + lb_1 < 0,\\
    \deg{\sheaf{N_{C,Z}}(nE-lC)} &= -b_1 - nx_1b_1 + nx_2 + lb_1 = -b_1 - na_1x_0 + lb_1 < 0
    \end{align*}
for $0 \le l \le \alpha-2$. Therefore $h^0(\sheaf{T_Z}(nE-lC)|_C)=0$. By using $\sheaf{T_Z}|_C = \sheaf{T_C} \oplus \sheaf{N_{C,Z}}$ and the Riemann-Roch, we have $h^1(\sheaf{T_Z}(nE-lC)|_C)=2na_1x_0-2lb_1+b_1$. Hence the claim follows.

From the decomposition sequence
    \[0 \to \sheaf{T_Z}(nE-(l+1)C)|_{(nx-l-1)C} \to \sheaf{T_Z}(nE-lC)|_{(nx-l)C} \to \sheaf{T_Z}(nE-lC)|_C \to 0\]
and the above claim, we have
    \[h^1(\sheaf{T_Z}(nE-lC)|_{(nx-l)C}) = h^1(\sheaf{T_Z}(nE-lC)|_C) + h^1(\sheaf{T_Z}(nE-(l+1)C)|_{(nx-l-1)C})\]
for $0 \le l \le \alpha -1$. Therefore
    \begin{equation*}
    \begin{split}
    h^1(\sheaf{T_Z}(nE)|_{nxC}) &= h^1(\sheaf{T_Z}(nE)|_C) + h^1(\sheaf{T_Z}(nE-C)|_{(nx-1)C})\\
    &=\cdots\\
    &= \sum_{l=0}^{\alpha-1} h^1(\sheaf{T_Z}(nE-lC)|_C) + h^0(\sheaf{T_Z}(nE-\alpha C_1)|_{(nx_1-\alpha)C_1})\\
    &= \sum_{l=0}^{\alpha-1}(2na_1x_0-2lb_1+b_1) + h^0(\sheaf{T_Z}(nE-\alpha C_1)|_{(nx_1-\alpha)C_1})\\
    &\sim O(n^2)
    \end{split}
    \end{equation*}
if $n \gg 0$ because $\alpha \sim O(n)$ for $n \gg 0$.
\end{proof}

\begin{proposition}\label{proposition:grows_quadrically}
The dimension $h^1(Z, \sheaf{T_Z}(nE))$ grows at least quadratically in $n \gg 0$.
\end{proposition}

\begin{proof}
Consider the exact sequence
    \begin{equation*}
    \cdots \to H^1(\sheaf{T_Z}(nE)) \xrightarrow{\alpha} H^1(nE, \sheaf{T_Z}(nE) \otimes \sheaf{O_{nE}}) \xrightarrow{\beta} H^2(Z, \sheaf{T_Z}) \to \cdots
    \end{equation*}
induced from the exact sequence
    \begin{equation*}
    0 \to \sheaf{T_Z} \to \sheaf{T_Z}(nE) \to \sheaf{T_Z}(nE) \otimes \sheaf{O_{nE}} \to 0.
    \end{equation*}
Then we have
    \begin{equation*}
    h^1(nE, \sheaf{T_Z}(nE) \otimes \sheaf{O_{nE}}) = \dim \image{\alpha} + \dim \image{\beta}.
    \end{equation*}
By Lemma~\ref{lemma:h^1(T_Z(nE)|nE)}, the left hand side of the above equation grows quadratically for $n \gg 3$, but $\dim \image{\beta}$ bounds for $h^2(Z, \sheaf{T_Z})$. Hence $\dim \image{\alpha}$ must grows quadratically for $n \gg 0$, which implies that $h^1(Z, \sheaf{T_Z}(nE))$ grows at least quadratically in $n \gg 0$.
\end{proof}

The infinitesimal deformation spaces $H^1(Z, \sheaf{T_Z}(nE))$ ($n \gg 0$) form a stratification of $H^1(Y, \sheaf{T_Y})$; Proposition~\ref{proposition:injective}.

\begin{lemma}\label{lemma:H^0(T_Z(nE-lC_1)|C_1)=0}
For each pair $(i,j)$, we have $H^0(x_{ij}C_{ij}, \sheaf{T_Z}(nE) \otimes \sheaf{O_{x_{ij}C_{ij}}}) = 0$ for every $n \gg 0$.
\end{lemma}

\begin{proof}
We only prove the lemma in case $i=1$, $j=1$, and $m_1 \ge 2$. The proof of the other cases are similar. For simplicity we denote  $C_{1j}$ by $C_j$, $b_{1j}$ by $b_j$, $m_1$ by $m$, and $x_{1j}$ by $x_j$.

According to the claim in the proof of Lemma~\ref{lemma:h^1(T_Z(nE)|nE)}, we have $H^0(C_1, \sheaf{T_Z}(nE-lC_1)|_{C_1})=0$ for $0 \le l \le \left[ \frac{na_1x_0+2}{b_1}\right] -1$. Note that $x_1 \le \left[ \frac{na_1x_0+2}{b_1}\right] + 1$ for $n \gg 0$. Therefore
    \begin{equation}\label{equation:H^0(T_Z(nE-lC_1)|C_1)=0}
    H^0(C_1, \sheaf{T_Z}(nE-lC_1)|_{C_1})=0
    \end{equation}
for $0 \le l \le x_1-2$. Consider the decomposition sequence
    \begin{multline*}
    0 \to \sheaf{T_Z}(nE-(l+1)C_1)|_{(x_1-l-1)C_1} \\
    \to \sheaf{T_Z}(nE-lC_1)|_{(x_1-l)C_1} \\
    \to \sheaf{T_Z}(nE-lC_1)|_{C_1} \to 0.
    \end{multline*}
By \eqref{equation:H^0(T_Z(nE-lC_1)|C_1)=0} and the induction on $l$, it follows that $H^0(\sheaf{T_Z}(nE)|_{x_1C_1})=0$.
\end{proof}

From the exact sequence
    \begin{equation*}
    0 \to \sheaf{T_Z}(nE) \to \sheaf{T_Z}((n+1)E) \to \sheaf{T_Z}((n+1)E) \otimes \sheaf{O_E} \to 0,
    \end{equation*}
we have an induced map $H^1(Z, \sheaf{T_Z}(nE)) \to H^1(Z, \sheaf{T_Z}((n+1)E))$.

\begin{proposition}\label{proposition:injective}
The map $H^1(Z, \sheaf{T_Z}(nE)) \to H^1(Z, \sheaf{T_Z}((n+1)E))$ is injective for every $n \gg 0$.
\end{proposition}

\begin{proof}
Let $E_i = \sum_{j=1}^{m_j} x_{ij}C_{ij}$. Since the supports of $D_i$'s are disjoint we have $H^1(Z, \sheaf{T_Z}(nE)) = \oplus_{i=1}^{k} H^1(Z, \sheaf{T_Z}(nE_i))$. Therefore it is enough to show that $H^1(Z, \sheaf{T_Z}(nE_i)) \to H^1(Z, \sheaf{T_Z}((n+1)E_i))$ is injective for every $n \gg 0$. We prove only the case of $i=1$ and $m_i \ge 2$. For simplicity we denote  $C_{1j}$ by $C_j$, $b_{1j}$ by $b_j$, $m_1$ by $m$, and $x_{1j}$ by $x_j$.

Consider the exact sequence
    \begin{equation*}
    0 \to \sheaf{T_Z}(nE-x_1C_1)|_{E-x_1C_1} \to \sheaf{T_Z}(nE)|_E \to \sheaf{T_Z}(nE)|_{x_1C_1} \to 0.
    \end{equation*}
By Lemma~\ref{lemma:H^0(T_Z(nE-lC_1)|C_1)=0}, we have $H^0(\sheaf{T_Z}(nE)|_{x_1C_1})=0$; hence,
    \begin{equation}\label{equation:H^0(T_Z(nE)|E)}
    H^0(\sheaf{T_Z}(nE)|_E) = H^0(\sheaf{T_Z}(nE-x_1C_1)|_{E-x_1C_1}) \subset H^0(\sheaf{T_Z}(nE)|_{E-x_1C_1}).
    \end{equation}
Consider the exact sequence
    \begin{equation*}
    0 \to \sheaf{T_Z}(nE-x_2C_2)|_{E-x_1C_1-x_2C_2} \to \sheaf{T_Z}(nE)|_{E-x_1C_1} \to \sheaf{T_Z}(nE)|_{x_2C_2} \to 0.
    \end{equation*}
By Lemma~\ref{lemma:H^0(T_Z(nE-lC_1)|C_1)=0}, we have $H^0(\sheaf{T_Z}(nE)|_{x_2C_2})=0$. Therefore it follows from \eqref{equation:H^0(T_Z(nE)|E)} that
    \begin{equation*}
    \begin{split}
    H^0(\sheaf{T_Z}(nE)|_E) &\subset H^0(\sheaf{T_Z}(nE)|_{E-x_1C_1}) = H^0(\sheaf{T_Z}(nE-x_2C_2)|_{E-x_1C_1-x_2C_2}) \\
    & \subset H^0(\sheaf{T_Z}(nE)|_{E-x_1C_1-x_2C_2}).
    \end{split}
    \end{equation*}
Repeating this process, then we have
    \begin{equation*}
    H^0(\sheaf{T_Z}(nE)|_E) \subset H^0(\sheaf{T_Z}(nE)|_{E-x_1C_1}) \subset \dotsb \subset H^0(\sheaf{T_Z}(nE)|_{x_mC_m}) = 0. \qedhere
    \end{equation*}
\end{proof}

\section{Properties of the infinitesimal deformations}
\label{section:nontrivial}

A small deformation of the domain $Y$ in $Z$ changes the complex structure only near the boundary of $Y$ while keeping unchanged the complex structure away from $\partial Y$. In this section we will show that the deformations induced from $H^1(Z, \sheaf{T_Z}(nE))$ are not small deformations. In fact, any non-zero $\alpha \in H^1(Z, \sheaf{T_Z}(nE))$ induces a non-trivial infinitesimal deformation of a certain curve away from $\partial Y$; Theorem~\ref{theorem:nontrivial}.

Since the divisor $L$ in Lemma~\ref{lemma:E-and-L} descends to an ample divisor on $X$, we may choose an irreducible smooth curve $C \in \linsys{nL}$ ($n \gg 0$) on $Z$ such that $C \cap U \neq \varnothing$.

\begin{theorem}\label{theorem:nontrivial}
Any infinitesimal deformation induced from a non-zero element $H^1(Z, \sheaf{T_Z}(nE))$ preserves $C \cap Y$ but changes infinitesimally the complex structure of $C \cap Y$.
\end{theorem}

\begin{proof}
We use a similar method in Takamura~\cite{Takamura-1996}. Let $C'=C \cap Y$. Consider the exact sequence
    \begin{equation*}
    H^1(Y, \sheaf{T_Y}(-\log{C'})) \xrightarrow{\psi} H^1(Y, \sheaf{T_Y}) \to H^1(C', \sheaf{N_{C',Y}})
    \end{equation*}
induced from
    \begin{equation*}
    0 \to \sheaf{T_Y}(-\log{C'}) \to \sheaf{T_Y} \to \sheaf{N_{C',Y}} \to 0.
    \end{equation*}
Since $C \cap U \neq \varnothing$, $C'$ is stein; hence, $H^1(C', \sheaf{N_{C',Y}})=0$ and the map $\psi$ is surjective. Therefore any infinitesimal deformation of $Y$ induced from a non-zero element in $H^1(Z, \sheaf{T_Z}(nE))$ preserves $C'$.

Let $\alpha \in H^1(Z, \sheaf{T_Z}(nE))$ be a non-zero element. We denote again by $\alpha$ the image of $\alpha$ of the injective map $H^1(Z, \sheaf{T_Z}(nE)) \to H^1(Y, \sheaf{T_Y})$. Take any $\widetilde{\alpha} \in H^1(Y, \sheaf{T_Y}(-\log{C'}))$ such that $\psi(\widetilde{\alpha}) = \alpha$. Consider the commutative diagram
    \begin{equation*}
    \xymatrix{
    H^1(Y, \sheaf{T_Y}(-\log{C'})) \ar[r]^{\psi} \ar[d]_{\phi} & H^1(Y, \sheaf{T_Y}) \ar[d] \\
    H^1(C', \sheaf{T_{C'}}) \ar[r] & H^1(C', \sheaf{T_Y}|_{C'})
    }
    \end{equation*}
In order to prove that $\widetilde{\alpha}$ changes infinitesimally the complex structure of $C'$, it is enough to show that the image of $\widetilde{\alpha}$ by $\phi : H^1(Y, \sheaf{T_Y}(-\log{C'})) \to H^1(C', \sheaf{T_{C'}})$ is non-zero. In the proof of Proposition~\ref{theorem:infinitesimal_deformation}, we showed that the restriction $\widetilde{\alpha}|_C' \in H^1(C', \sheaf{T_Z}|_{C'})$ is non-zero because
    \begin{equation*}
    H^1(Z, \sheaf{T_Z}(nE)) \to H^1(C', \sheaf{T_Z}(nE)|_{C'}) \cong H^1(Z, \sheaf{T_Z}|_{C'})
    \end{equation*}
is injective for $n \gg 0$. Since the above diagram commutes, $\phi(\widetilde{\alpha})$ is also non-zero.
\end{proof}

\subsection*{Acknowledgements}

The author would like to thank Heesang Park for valuable discussions during the work.

\providecommand{\bysame}{\leavevmode\hbox to3em{\hrulefill}\thinspace}
\providecommand{\MR}{\relax\ifhmode\unskip\space\fi MR }
\providecommand{\MRhref}[2]{%
  \href{http://www.ams.org/mathscinet-getitem?mr=#1}{#2}
}
\providecommand{\href}[2]{#2}

\end{document}